\documentclass[12pt]{amsart}
\usepackage{amssymb}
\usepackage{mathrsfs}
\usepackage{graphicx}
\usepackage{color}
\newcommand\la{\lambda}

\newcommand\s{\sigma}

\newcommand\vf{\varphi}

\renewcommand\t{\tau}

\newcommand\Om{\Omega}

\newcommand\D{\Delta}

\newcommand\vL{\varLambda}

\newcommand\vG{\varGamma}

\newcommand{\ZZ}{\mathbb Z}

\newcommand{\CC}{\mathbb C}

\newcommand\bk{\mathbf{k}}
\newcommand\bh{\mathbf{h}}

\newcommand\CA{\mathcal{A}}
\newcommand\CB{\mathcal{B}}

\newcommand\CH{\mathcal{H}}

\newcommand\CO{\mathcal{O}}
\newcommand\CK{\mathcal{K}}
\newcommand\CP{\mathcal{P}}

\newcommand\FS{\mathfrak S}

\newcommand\Fm{\mathfrak m}

\newcommand{\Fk}{\mathfrak{k}}
\newcommand{\Fd}{\mathfrak{d}}

\newcommand\wh{\widehat}

\newcommand\ol{\overline}

\newcommand{\lan}{\langle}
\newcommand{\ran}{\rangle}

\newcommand{\ra}{\rightarrow }

\newcommand\Ker{\operatorname{Ker}}
\newcommand\Hom{\operatorname{Hom}}
\newcommand\End{\operatorname{End}}

\newcommand\Ind{\operatorname{Ind}}
\newcommand\res{\operatorname{res}}
\newcommand\Res{\operatorname{Res}}
\newcommand\Irr{\operatorname{Irr}}

\newcommand\Top{\operatorname{Top}}
\newcommand\Rad{\operatorname{Rad}}

\newcommand{\KZ}{\operatorname{KZ}}
\newcommand\GL{\operatorname{GL}}

\newcommand{\opp}{\operatorname{opp}}
\newcommand{\cmod}{\operatorname{-mod}}
\newcommand{\cproj}{\operatorname{-proj}}
\newcommand{\isom}{\,\raise2pt\hbox{$\underrightarrow{\sim}$}\,}
\newcounter{ichi}
\newcommand{\roi}{\roman{ichi}}
\newcounter{ni}
\newcommand{\roii}{\roman{ni}}
\newcounter{san}
\newcounter{yon}
\newcounter{go}
\newcounter{roku}
\newcounter{nana}
\newcommand{\Sc}{\mathscr{S}}

\newcommand{\He}{\mathscr{H}}
\newcommand{\A}{\mathscr{A}}
\newcommand{\B}{\mathscr{B}}

\newcommand{\dwar}{\hspace{-1mm} \downarrow} 
\newcommand{\upar}{\hspace{-1mm} \uparrow } 
\newtheorem{thm}{Theorem}[section]
\newtheorem{lem}[thm]{Lemma}
\newtheorem{cor}[thm]{Corollary}
\newtheorem{prop}[thm]{Proposition}

\def \para{\refstepcounter{thm} \par\medskip\noindent
                \textbf{\thethm .} }

\def \remark{\refstepcounter{thm} \par\medskip\noindent
                \textbf{Remark \thethm .} }

\def \remarks{\refstepcounter{thm} \par\medskip\noindent
                \textbf{Remarks \thethm .} }

\allowdisplaybreaks[4]
\numberwithin{equation}{thm}

\begin{document}
\setlength{\baselineskip}{4.9mm}
\setlength{\abovedisplayskip}{4.5mm}
\setlength{\belowdisplayskip}{4.5mm}
%%%
%%%
\renewcommand{\theenumi}{\roman{enumi}}
\renewcommand{\labelenumi}{(\theenumi)}
\renewcommand{\thefootnote}{\fnsymbol{footnote}}
%%%
\renewcommand{\thefootnote}{\fnsymbol{footnote}}
%\NoBlackBoxes
\parindent=20pt
%%%%%%%%%%%%%%%%%%%%
\newcommand{\dis}{\displaystyle}
%%%%%%%%%%%%%%%%%%%%%%%%%%%%%%%%%%%

\medskip
\begin{center}
{\large \bf Blocks of category $\CO$ for rational Cherednik algebras 
	\\[2mm]
	and of cyclotomic Hecke algebras of type $G(r,p,n)$} 
\\
\vspace{1cm}
Kentaro Wada\footnote{This research was  supported  by JSPS Research Fellowships for Young Scientists. 
} 

\address{Graduate School of Mathematics Nagoya University, 
	Furocho, Chikusaku, Nagoya, Japan 464-8602}
\email{kentaro-wada@math.nagoya-u.ac.jp} 

\end{center}

\title{}
\maketitle 
\markboth{Kentaro Wada}{Blocks of category $\CO$ for rational Cherednik algebras 
and of cyclotomic Hecke algebras}

%%%%%%%%%%%%%%%%%%%%%%%%%%%%%%%%%%%%%%%%%%%%%%%%%%%%%%%%%%%%%%%%%%%%%%%%%%%%%%%%%%%%%%%%%%%%%%%%%%%%%%%%%%%%%%%%%%%%%

%%%%%%%%%%%%%%%%%%%%%%%%%%%%%%%%%%%%%%%%%%%%%%%%%%%%%%%%%%%%%%%%%%%%%%%%%%%%%%%%%%%%%%%%%%%%%%%%%%%%%%%%%%%%%%%%%%%%%

%%%%%%%%%%%%%%%%%%%%%%%%%%%%%%%%%%%%%%%%%%%%%%%%%%%%%%%%%%%%%%%%%%%%%%%%%%%%%%%%%%%%%%%%%%%%%%%%%%%%%%

\begin{abstract}
We classify blocks of category $\CO$ for rational Cherednik algebras 
and of cyclotomic Hecke algebras of type $G(r,p,n)$ 
by using the \lq\lq residue equivalence" for multi-partitions. 
\end{abstract}

%%%%%%%%%%%%%%%%%%%%%%%%%%%%%%%%%%%%%%%%%%%%%%%%%%%%%%%%%%%%%%%%%%%%%%%%%%%%%%%%%%%%%%%%%%%%%%%%%%%%%%

\setcounter{section}{-1}

\section{Introduction} 
Let $V$ be a finite dimensional vector space over $\CC$, 
and $W \subset \GL(V)$ be a finite complex reflection group.   
The rational Cherednik algebra $\CH=\CH(W)$ over $\CC$ 
associated to $W$ was introduced by \cite{EG}. 
It is known that the category $\CO$ of $\CH$ 
is a highest weight category 
with standard modules $\{\D(\la) \,|\, \la \in \vL^+\}$, 
where $\vL^+$ is an index set of pairwise non-isomorphic simple $W$-modules over $\CC$ 
(\cite{Gu}, \cite{GGOR}).  
Let $\He=\He(W)$ be the cyclotomic Hecke algebra associated to $W$ 
with appropriate parameters. 
Let 
$\KZ : \CO \ra \He \cmod$ be the Knizhnik-Zamolodchikov functor defined in \cite{GGOR}. 
It is known that 
$\CO$ is a quasi-hereditary cover (highest weight cover) of $\CH$ 
in the sense of \cite{R}.  
Put $S(\la) = \KZ(\D(\la))$.  
We see that 
there exists a one-to-one correspondence between 
the blocks of $\CO$ and of $\CH$ 
thanks to the double centralizer property. 
Moreover, 
we see that 
the classification of blocks of $\CO$ and of $\CH$ 
is given by the linkage classes on $\{\D(\la) \,|\, \la \in \vL^+\}$ 
or on $\{S(\la)\,|\, \la \in \vL^+\}$ 
(see \S \ref{section quasi-hereditary} for detailes).    
Hence, 
it is enough to classify the blocks of $\CO$ and of $\He$ 
that we determine the linkage classes on $\{S(\la)\,|\, \la \in \vL^+\}$.

In the case where  $W$ is the complex reflection group of type $G(r,1,n)$, 
$\He$ is also called an Ariki-Koike algebra. 
In this case, 
$\vL^+$ is the set of $r$-partitions of size $n$ denoted by $\CP_{n,r}$.  
Then 
the linkage classes on $\{S(\la)\,|\, \la \in \vL^+\}$ 
is given by the equivalent relation \lq\lq $\sim_R$", so called residue equivalence, on $\CP_{n,r}$ 
by \cite{LM}.   
(Note that the Specht module $S^\la$ ($\la \in \vL^+$) 
being considered in \cite{LM} does not coincides with $S(\la)$ 
in general. However, one sees that 
the linkage classes on $\{S^\la\,|\, \la \in \vL^+\}$ coincides 
with the linkage classes on $\{S(\la)\,|\, \la \in \vL^+\}$. See \S \ref{G(r,1,n)}.) 

Our purpose is to classify the blocks of $\CO$ and of $\He$ 
in the case where $W$ is the complex reflection group of type $G(r,p,n)$. 
As seen in the above, we should determine the linkage classes on $\{S(\la)\,|\, \la \in \vL^+\}$. 
Let $W^\dag$ be the complex reflection group of type $G(r,1,n)$, 
and we denote by adding the superscript $\dag$ for objects of type $G(r,1,n)$. 
It is known that 
$W$ is a normal subgroup of $W^\dag$ with the index $p$, 
and that 
$\He$ is a subalgebra of $\He^\dag$. 
An index set $\vL^+$ of pairwise non-isomorphic simple $W$-modules over $\CC$ 
(thus, $\vL^+$ is also an index set of standard modules of $\CO$) 
is given as the equivalent classes of $\CP_{n,r} \times \ZZ / p \ZZ$ 
under a certain equivalent relation \lq\lq $\sim_{\ast}$" on $\CP_{n,r} \times \ZZ / p \ZZ$ 
(see \ref{par res D to B} for details). 
We denote by $\la \lan i \ran  \in \vL^+$  
the equivalent class containing $(\la, \ol{i}) \in \CP_{n,r} \times \ZZ / p\ZZ$.

Some relations between representations of $\He$ and of $\He^\dag$ 
have been studied in \cite{A95}, \cite{GJ}, 
\cite{Hu02}, \cite{Hu03}, \cite{Hu04}, \cite{Hu07} and \cite{Hu09}  
by using Clifford theory. 
Combining these results with some fundamental properties of quasi-hereditary covers,   
and with the classification of blocks of $\He^\dag$ by using the residue equivalence \lq\lq $\sim_R$", 
we give the classification of the blocks of $\CO$ and of $\He$ 
by using a certain equivalent relation \lq\lq $\approx $" on $\CP_{n,r}$ 
as follows.

Let \lq\lq $\approx $" be the equivalent relation on $\CP_{n,r}$ 
defined by 
$\la \approx \mu$ if 
$\la \sim_R \mu[j]$ for some $j \in \ZZ$, 
where 
$\mu[j] \in \CP_{n,r}$ 
is obtained from $\mu \in \CP_{n,r}$ by a certain permutation of components of $\mu$ 
(see \ref{par res D to B} for the precise definition of $\mu[j]$).  
Put 
$\vG=\{ \la \in \CP_{n,r} \,|\, \la \not\sim_R \mu \text{ for any } \mu \in \CP_{n,r} 
	\text{ such that } \mu \not=\la \}$. 
Then our main theorem is the following. 
\\[2mm]
\textbf{Theorem \ref{main thm}} 
\begin{enumerate}
\item 
If $\la \in \vG$, 
then 
$\D(\la \lan i \ran)$ 
(resp. $S(\la \lan i \ran)$) 
is a simple object of $\CO$ 
(resp. simple $\He$-module) 
for any $i \in \ZZ$. 
Moreover, 
$\D(\la \lan i \ran)$ 
(resp. $S(\la \lan i \ran)$) 
is a block of $\CO$ (resp. of $\He$) 
itself. 

\item 
For $\la,\mu \in \CP_{n,r} \setminus \vG$ and $i,j \in \ZZ$, 
we have that 
\begin{align*}
&\text{
Both of $\D(\la \lan i \ran)$ and $\D(\mu \lan j \ran)$ 
belong to the same block of $\CO$
} 
\\
&\Leftrightarrow \text{ 
Both of $S(\la \lan i \ran)$ and $S(\mu \lan j \ran)$ 
belong to the same block of $\He$
}
\\
&\Leftrightarrow 
\la \approx \mu. 
\end{align*}

\end{enumerate}

%%%%%%%%%%%%%%%%%%%%%%%%%%%%%%%%%%%%%%%%%%%%%%%%%%%%%%%%%%%%%%%%%%%%%%%%%%%%%%%%%%%%%%%%%%%%%%%%%%%%%%
\quad \\
\textbf{Acknowledgment:} 
The author is grateful to Professors 
S. Ariki, T. Kuwabara, H. Miyachi and T. Shoji for 
many valuable discussions and comments.

%%%%%%%%%%%%%%%%%%%%%%%%%%%%%%%%%%%%%%%%%%%%%%%%%%%%%%%%%%%%%%%%%%%%%%%%%%%%%%%%%%%%%%%%%%%%%%%%%%%%%%
\quad \\
\textbf{Notations:} 
For an algebra $\A$, 
we denote by $\A \cmod$ the category of finitely generated $\A$-modules,  
and 
denote by 
$\A \cproj$ the full subcategory of $\A \cmod$ 
consisting of projective objects.  
Let 
$K_0(\A \cmod)$ be a Grothendieck group of $\A \cmod$. 
We denote by $[M]$ the image of $M$ in the $K_0(\A \cmod)$ for $M \in \A \cmod$. 
For $M \in \A \cmod$ and simple object $L$ of $\A \cmod$, 
we denote by 
$[M:L]_\A$ 
the multiplicity of $L$ in the composition series of $M$. 
We also denote by 
$\A^{\opp}$ the opposite algebra of $\A$. 
%

%%%%%%%%%%%%%%%%%%%%%%%%%%%%%%%%%%%%%%%%%%%%%%%%%%%%%%%%%%%%%%%%%%%%%%%%%%%%%%%%%%%%%%%%%%%%%%%%%%%%%%%%%%%%%%%%%%%%%

%%%%%%%%%%%%%%%%%%%%%%%%%%%%%%%%%%%%%%%%%%%%%%%%%%%%%%%%%%%%%%%%%%%%%%%%%%%%%%%%%%%%%%%%%%%%%%%%%%%%%%%%%%%%%%%%%%%%%
\section{Some properties of quasi-hereditary covers} 
\label{section quasi-hereditary}
In this section, 
we recall some notions of quasi-hereditary covers 
from \cite{R}, 
and prepare some fundamental properties. 

\para
Let $\A$ be a quasi-hereditary algebra over a field. 
Take a projective object $P$ in $\A \cmod$, 
and 
put 
$\B=\End_{\A}(P)^{\opp}$.  
Then we have the exact functor 
$F=\Hom_{\A}(P,-) : \A \cmod \ra \B \cmod$. 
Let 
$X$ 
be a progenerator of 
$\A \cmod$ 
such that 
$X=P \oplus P'$ 
for some projective object $P'$ in $\A \cmod$. 
Then 
$\End_{\A}(X)^{\opp}$ is Morita equivalent to $\A$. 
We may suppose that 
$\End_{\A}(X)^{\opp} \cong \A$ 
without loss of generality,  
and 
we suppose that $\End_{\A}(X)^{\opp} \cong \A$. 

Throughout this section, 
we assume the following condition.  
\begin{description}
\item[(A1)]
The functor $F$ is fully faithful 
when we restrict to $\A \cproj$. 
\end{description}
Hence, 
$\A$ is a quasi-hereditary cover of $\B$ 
in the sense of \cite{R}. 
Since 
$X \in \A \cproj$, 
we have that 
\begin{align}
\label{A commutant of B}
\A \cong \End_\A(X)^{\opp} \cong \End_{\B}(F(X))^{\opp}.
\end{align} 

Note that 
$X=P \oplus P'$. 
Let  
$\vf_P^o \in \End_{\A}(X)$ 
be such that 
$\vf_P^o$ is the identity map on $P$,    
and $0$-map on $P'$.
We denote by 
$\vf_P$ 
the  element of $\A \cong \End_{\A}(X)^{\opp}$ corresponding to $\vf_P^o$. 
It is clear that 
$\vf_P$ is an idempotent. 
Since 
\begin{align*}
F(X) 
&
\cong \Hom_{\A}(P,P) \oplus \Hom_{\A}(P,P') 
\\
&
\cong \End_{\A}(X) \vf_P^o
\\
&
\cong \vf_P \A
\end{align*}
as right $\A$-modules, 
we have   following isomorphisms of algebras: 
\begin{align*}
\End_{\A^{\opp}}(F(X)) 
&
\cong \End_{\A^{\opp}}(\vf_p \A)
\\
& 
\cong \vf_P \A \vf_P 
\\
&
\cong \big( \vf_P^o \End_{\A}(X) \vf_P^o \big)^{\opp} 
\\
&
\cong \End_{\A}(P)^{\opp}
\\
&
= \B.
\end{align*}

Thus, 
we have the double centralizer property: 
\begin{align}
\label{double centralizer property}
\A \cong \End_{\B}(F(X))^{\opp}, 
\quad 
\B \cong \End_{\A^{\opp}}(F(X)).
\end{align}
This double centralizer property 
implies the isomorphism 
$Z(\A) \ra Z(\B)$, 
where 
$Z(\A)$ (resp. $Z(\B)$)  
is the center of $\A$ (resp. $\B$). 
Thus, 
there exists a bijection 
between blocks of 
$\A$ and of $\B$.

\para 
Recall that 
$\A$ is a quasi-hereditary algebra. 
Let 
$\{\D(\la) \,|\, \la \in \vL^+\}$ 
be the set of standard modules, 
and 
$\{\nabla (\la) \,|\, \la \in \vL^+\}$ 
be the set of costandard modules 
of $\A$. 
For  
$\la \in \vL^+$, 
let 
$L(\la)$ be 
the unique simple top of $\D(\la)$, 
and 
$P(\la)$ 
be the projective cover of $L(\la)$. 
Then 
$\{L(\la)\,|\, \la \in \vL^+\}$ 
gives a complete set of non-isomorphic 
simple $\A$-modules. 

For 
$\la \in \vL^+$, 
put 
$S(\la) = F(\D(\la))$, 
$D(\la) = F(L(\la))$ 
and 
$\vL^+_0 =\{\la \in \vL^+\,|\, D(\la) \not=0\}$.  
Since 
$\B \cong \vf_P \A \vf_P$ 
and 
$F=\Hom_{\A}(P,-) =\Hom_{\A}(\A \vf_P,-)$, 
the following lemma is well-known 
(see e.g. \cite[Appendix]{Don}). 

\begin{lem}\
\label{lem projective simple B}
\begin{enumerate}
\item 
For $\la \in \vL^+_0$, 
we have 
$D(\la) \cong \Top F(P(\la)) \cong \Top S(\la)$. 

\item 
$\{F(P(\la)) \,|\, \la \in \vL_0^+ \}$ 
gives a complete set of non-isomorphic indecomposable projective $\B$-modules.

\item 
$\{D(\la) \,|\, \la \in \vL_0^+ \}$ 
gives a complete set of non-isomorphic simple $\B$-modules. 
\end{enumerate}
\end{lem}

\para 
For $\la, \mu \in \vL^+$, 
we denote by 
$P(\la) \sim P(\mu)$ 
if there exists a sequence 
$\la =\la_1, \la_2, \cdots, \la_{k+1}=\mu$ ($\la_i \in \vL^+$) 
such that 
$P(\la_i)$ and $P(\la_{i+1})$ have a common composition factor 
for any $i=1,\cdots,k$. 
Then 
\lq\lq $\sim$" gives an equivalent relation on 
$\{P(\la) \,|\, \la \in \vL^+\}$. 
It is well-known that 
$P(\la) \sim P(\mu)$ if and only if $P(\la)$ and $P(\mu)$ belong to the same block of $\A$. 
Similarly, 
we define an equivalent relation \lq\lq $\sim$" on 
$\{F(P(\la)) \,|\, \la \in \vL^+_0 \}$, 
and we have that  
$F(P(\la)) \sim F(P(\mu))$ 
if and only if 
$F(P(\la))$ and $F(P(\mu))$ belong to the same block of $\B$. 
We have the following lemma.

\begin{lem}
\label{lem P la sim P mu if and only if F P la sim P mu} 
For $\la, \mu \in \vL^+_0$, 
we have 
\[ P(\la) \sim P(\mu) 
\quad \text{ if and only if } 
\quad 
F (P(\la)) \sim F (P(\mu)). 
\]
\end{lem}

\begin{proof}
Let 
$1_\A$ be the unit element of $\A$, 
and 
$1_\A =e_1 + \cdots + e_k$ 
be the decomposition 
of $1_\A$ 
to the sum of primitive central idempotents. 
Then we have the decomposition 
$X=X_1 \oplus \cdots \oplus X_k$, 
where 
$X_i = e_i X$ 
belong to a block of $\A$. 
Then we have the decomposition of $\A$ to blocks: 
\begin{align}
\label{block decom A}
\A 
&\cong \End_{\B} (F(X))^{\opp}  
\\
&
=\End_{\B}(F(X_1))^{\opp} \oplus \cdots \oplus \End_{\B}(F(X_k))^{\opp}
\notag
\end{align} 
since 
$\Hom_{\B}(F(X_i),F(X_j)) \cong \Hom_{\CA}(X_i,X_j) =0$ 
if $i \not=j$. 
Assume that 
\[
F(X_i)\cong C_1 \oplus C_2 
\text{ as $\B$-modules}
\] 
such that 
$C_1$ and $C_2$ have no common composition factor. 
Then we have that 
\[
\End_{\B}(F(X_i)) = \End_{\B}(C_1) \oplus \End_{\B}(C_2) 
\text{ as two sided ideals of $\A$}.
\] 
However, this contradicts to the decomposition 
\eqref{block decom A}. 
Thus we have that 
all composition factors of $F(X_i)$ belong to the same block of $\CB$. 
This implies that 
\[ 
P(\la) \sim P(\mu) \text{ only if } F(P(\la)) \sim F(P(\mu)) \quad \text{ for } \la,\mu \in \vL^+_0.
\]

On the other hand,  
for $\la, \mu \in \vL^+_0$, 
assume that 
$F(P(\la))$ and $F(P(\mu))$ 
have a common composition factor 
$D(\nu)$. 
Then we have 
\begin{align*}
&\dim \Hom_{\B}\big( F (P(\nu)) , F (P(\la)) \big) = [F (P(\la)) : D(\nu)] \not=0, \notag \\
&\dim \Hom_{\B}\big( F (P(\nu)) , F (P(\mu)) \big) = [F (P(\mu)) : D(\nu)] \not=0.  \notag
\end{align*}
By
the assumption (A1), 
we have that 
\begin{align*}
&\Hom_{\A}\big( P(\nu) , P(\la) \big) = \Hom_{\B}\big( F (P(\nu)) , F (P(\la)) \big) \not=0, \notag \\
&\Hom_{\A}\big( P(\nu) , P(\mu) \big) = \Hom_{\B}\big( F (P(\nu)) , F (P(\mu)) \big)  \not=0.  \notag
\end{align*}
Thus, we have that 
\begin{align*} 
P(\la) \sim P(\mu) \quad \text{if} \quad F(P(\la)) \sim F(P(\mu)) \quad \text{ for } \la,\mu \in \vL^+_0.
\end{align*}
Now we proved the Lemma.
\end{proof}

\remark 
In the proof of the above lemma, 
we also showed that,  
for $\la \in \vL^+ \setminus \vL^+_0$, 
a composition factor of $F(P(\la))$ 
is isomorphic to 
$D(\mu)$ 
for some $\mu \in \vL^+_0$ 
such that 
$P(\la) \sim P(\mu)$.

As a corollary, we have the following. 

\begin{cor}
\label{cor S la belong to block}
For each $\la \in \vL^+$, 
all composition factors of $S(\la)$ belong to the same block of $\B$. 
\end{cor}

\begin{proof}
Since 
$\D(\la)$ 
is an indecomposable $\A$-module, 
all composition factors of $\D(\la)$ 
belong to the same block of $\A$. 
Thus, 
Lemma  \ref{lem P la sim P mu if and only if F P la sim P mu} 
implies the corollary 
since 
a composition series of $S(\la)$ as $\B$-modules 
is obtained from 
a composition series of $\D(\la)$ as $\A$-modules 
by applying the functor $F$. 
\end{proof}

%

%

%\remark 
%If $\A$ is a $0$-faithful quasi-hereditary cover of $\B$, 
%i.e. 
%$\Hom_{\A}(M,N) \cong \Hom_{\B}(F(M), F(N))$ 
%for any $\D$-filtered $\CA$-modules $M,N$, 
%then 
%Corollary  \ref{cor S la belong to block}
%follows 
%from the fact 
%$\End_{\A}(\D(\la)) \cong \End_{\B}(S(\la))$ 
%immediately 
%(note that $\D(\la)$ is indecomposable). 

\para 
From now on, 
we assume the following additional condition: 
\begin{description}
\item[(A2)] 
$[\D(\la)] = [\nabla (\la)]$ in $K_0(\A \cmod)$ for any $\la \in \vL^+$. 
\end{description}

By the general theory of quasi-hereditary algebras, 
for $\la \in \vL^+$, 
$P(\la)$ 
has a $\D$-filtration 
such that 
$(P(\la) : \D(\mu)) = [\nabla (\mu) : L(\la)]_{\A}$, 
where 
$(P(\la) : \D(\mu))$ 
is the multiplicity of $\D(\mu)$ 
in a $\D$-filtration of $P(\la)$. 
Combining with the assumption (A2), 
we have that 
\begin{align}
\label{P la De mu = De mu L la in A}
(P(\la) : \D(\mu)) = [\D(\mu) : L(\la)]_{\A}.
\end{align}
This implies the following lemma.

\begin{lem}
\label{lem B-H reciprocity in B} 
For $\la, \mu \in \vL^+_0$, we have 
\[ 
[F (P(\la)) : D(\mu)]_{\B} =\sum_{\nu \in \vL^+} [S(\nu) : D(\la)]_{\B} [ S(\nu) : D(\mu)]_{\B}.
\]
\end{lem} 

\begin{proof}
Let 
\[ 
P(\la) =P_k \supset P_{k-1} \supset \cdots \supset P_{1} \supset P_0=0 
\,\, \text{ such that }\,\,  P_i / P_{i-1} \cong \D(\nu_i) 
\]
be a $\D$-filtration of $P(\la)$. 
By applying the functor $F$ to this filtration, 
we have a filtration of $F(P(\la))$ (as $\B$-modules) 
\[ 
F(P(\la)) = F(P_k) \supset F(P_{k-1}) \supset \cdots \supset F(P_{1}) \supset F(P_0)=0 
\]
such that 
$F (P_i) / F (P_{i-1}) \cong S(\nu_i)$. 
In this filtration, 
for $\nu \in \vL^+$, 
$S(\nu)$ appears $[\D(\nu) : L(\la)]_{\A} = [S(\nu) : D(\la)]_{\B}$ times 
by \eqref{P la De mu = De mu L la in A}. 
Thus, 
we have 
\[ 
[ F (P(\la)) : D(\mu) ]_{\B} = \sum_{\nu \in \vL^+} [S(\nu) : D(\la)]_{\B} [S(\nu) : D(\mu)]_{\B}. 
\\[-2.5em]
\]
\end{proof}

\para 
For $\la, \mu \in \vL^+$, 
we denote by 
$S(\la) \sim S(\mu)$ 
if there exists a sequence 
$\la =\la_1, \la_2, \cdots, \la_{k+1}=\mu$ ($\la_i \in \vL^+$) 
such that 
$S(\la_i)$ and $S(\la_{i+1})$ have a common composition factor 
for any $i=1,\cdots,k$. 
Then 
\lq\lq $\sim$" gives an equivalent relation on 
$\{S(\la) \,|\, \la \in \vL^+\}$. 
Similarly, 
we define an equivalent relation \lq\lq $\sim$" 
on $\{ \D(\la) \,|\, \la \in \vL^+\}$. 
Note that 
all composition factors of $\D(\la)$ 
belong to the same block of $\A$ 
since 
$\D(\la)$ is indecomposable.  

Corollary \ref{cor S la belong to block} 
and 
Lemma \ref{lem B-H reciprocity in B}  
imply the following proposition. 

\begin{prop}
\label{prop Sla Smu block}
For $\la, \mu \in \vL^+$ we have the following. 
\begin{enumerate}
\item   
$S(\la) \sim S(\mu)$ if and only if  
$S(\la)$ and $S(\mu)$  
belong to the same block of $\B$. 
\item 
$\D(\la) \sim \D(\mu)$ if and only if  
$\D(\la)$ and $\D(\mu)$  
belong to the same block of $\A$. 
\end{enumerate}
\end{prop}

\begin{proof}
(\roi) 
The \lq \lq only if " part is clear by Corollary \ref{cor S la belong to block}. 
We prove the \lq \lq if " part. 
Suppose that 
$S(\la)$ and $S(\mu)$ belong to the same block of $\B$. 
Let 
$D(\la')$ (resp. $D(\mu')$) 
be a composition factor of 
$S(\la)$ (resp. $S(\mu)$). 
Then 
$D(\la')$ and $D(\mu')$ 
belong to the same block of $\B$.  
This implies that 
$F(P(\la')) \sim F(P(\mu'))$. 
Thus 
there exists a sequence 
$\la'=\la_1, \la_2, \cdots, \la_{k+1}=\mu'$ ($\la_i \in \vL^+_0$) 
such that 
$\B (P(\la_i))$ and $\B (P(\la_{i+1}))$ 
have a common composition factor 
$D(\nu_i)$ 
for any $i=1,\cdots,k$. 
By Lemma \ref{lem B-H reciprocity in B}, 
this implies that 
there exist 
\begin{align*}
&\t_i \in \vL^+ 
\text{ such that } [ S(\t_i) : D(\la_i)]_{\B} \not=0 \text{ and } [S(\t_i) : D(\nu_i)]_{\B} \not=0, \\
&
\t'_i \in \vL^+ 
\text{ such that } [ S(\t'_i) : D(\la_{i+1})]_{\B} \not=0 \text{ and } [S(\t'_i) : D(\nu_i)]_{\B} \not=0
\end{align*}
for any $i=1,\cdots,k$. 
Thus we have 
\[ S(\la_i) \sim S(\t_i) \sim S(\nu_i) \sim S(\t'_i) \sim S(\la_{i+1})\]
for any $i=1,\cdots,k$.
This implies that 
\[ S(\la) \sim S(\la')=S(\la_1) \sim S(\la_2) \sim \cdots \sim S(\la_{k+1}) = S(\mu') \sim S(\mu).\]

(\roii) is proven in a similar way by using \eqref{P la De mu = De mu L la in A} instead of 
Lemma \ref{lem B-H reciprocity in B}. 
\end{proof}

\para 
From now on, 
we assume the following additional condition: 
\begin{description}
\item[(A3)] 
$S(\la) = F(\D(\la)) \not=0$ for any $\la \in \vL^+$.
\end{description}

Thanks to Proposition \ref{prop Sla Smu block},  
we can classify blocks of $\B$ (resp. blocks of $\A$) 
by equivalent classes of $\{S(\la) \,|\, \la \in \vL^+\}$ 
(resp. $\{\D(\la) \,|\, \la \in \vL^+\}$) 
with respect to the relation \lq \lq $\sim$". 
The following proposition gives a relation between blocks of $\A$ and of $\B$.

\begin{prop}
\label{prop equivalent D la sim D mu S la sim S mu}
For $\la, \mu \in \vL^+$, 
we have 
\[ \D(\la) \sim \D(\mu) \,\, \text{ if and only if } S(\la) \sim S(\mu) .\]
\end{prop}

\begin{proof}
By Lemma \ref{lem P la sim P mu if and only if F P la sim P mu}, 
we have that 
$\D(\la)$ and $\D(\mu)$ belong to the same block of $\A$ 
if and only if 
$S(\la)$ and $S(\mu)$ belong to the same block of $\B$. 
Thus Proposition \ref{prop Sla Smu block} implies this proposition. 
\end{proof}

%%%%%%%%%%%%%%%%%%%%%%%%%%%%%%%%%%%%%%%%%%%%%%%%%%%%%%%%%%%%%%%%%%%%%%%%%%%%%%%%%%%%%%%%%%%%%%%%%%%%%%%%%%%%%%%%%%%%%

%%%%%%%%%%%%%%%%%%%%%%%%%%%%%%%%%%%%%%%%%%%%%%%%%%%%%%%%%%%%%%%%%%%%%%%%%%%%%%%%%%%%%%%%%%%%%%%%%%%%%%%%%%%%%%%%%%%%%

\section{Rational Cherednik algebras}
\label{RCA}
\para 
Let $V$ be a finite dimensional vector space over $\CC$,  
and 
$W \subset \GL(V)$ 
be a finite complex reflection group. 
Let $\CA$ be the set of reflecting hyperplanes of $W$, 
and 
$\CA/W$ be the set of $W$-orbits of $\CA$.   
For $H \in \CA$, 
let $W_H $ 
be the subgroup of $W$ 
fixing $H$ pointwise,  
and 
put 
$e_H=|W_H|$.  
Take a set 
\[ \Om = 
\{ k_{H,i} \in \CC \,|\, H \in \CA/W, \, 0 \leq i \leq e_H \text{ such that } k_{H,0}=k_{H,e_H}=0\}.\] 
Let 
$\CH$ 
be the rational Cherednik algebra 
associated to $W$ with parameters $\Om$ 
(see \cite[3.1]{GGOR} for definitions). 
By \cite{EG}, 
it is known that 
$\CH$ has the triangular decomposition 
\[ \CH \cong S(V^\ast ) \otimes_{\CC} \CC W \otimes_{\CC} S(V) \quad \text{ as vector spaces}, \]
where 
$S(V)$ (resp. $S(V^\ast)$) is the symmetric algebra of $V$ (resp. the dual space $V^\ast$), 
and 
$\CC W$ is the group ring of $W$ over $\CC$.  

Let 
$\CO$ be 
the category of finitely generated $\CH$-modules 
which are locally nilpotent for the action of $S(V) \setminus \CC$. 
Let 
$\Irr W =\{E^\la \,|\, \la \in \vL^+\}$ 
be a complete set of non-isomorphic simple $\CC W$-modules. 
For $\la \in \vL^+$,  
put 
\[ \D(\la) = \CH \otimes_{S(V) \rtimes W} E^\la,\]
where 
$S(V) \rtimes W \cong S(V) \otimes_{\CC} \CC W$ is a subalgebra of $\CH$, 
and we regard $E^\la$  as a $S(V) \rtimes W$-module 
through the natural surjection $S(V) \rtimes W \ra \CC W$.  
It is known that 
$\CO$ turns out to be a highest weight category with standard modules $\{\D(\la) \,|\, \la \in \vL^+\}$ 
(\cite{GGOR}, \cite{Gu}).  
Let 
$L(\la)$ 
be the unique simple top 
of $\D(\la)$, 
then 
$\{L(\la) \,|\, \la \in \vL^+\}$ 
is a complete set of non-isomorphic simple objects in $\CO$. 
For
$\la \in \vL^+$, 
we denote by 
$P(\la)$ 
the projective cover of $L(\la)$.

\para 
Let 
$\He$ 
be the cyclotomic Hecke algebra of $W$ corresponding $\CH$ 
(see \cite[5.2.5]{GGOR} for the choice of parameters).  
Then 
the Knizhnik-Zamolodchikov functor (simply, KZ functor)  
$\KZ : \CO \ra \He\text{-mod}$ 
is defined in \cite[5.3]{GGOR}. 
KZ functor is a exact functor,  
and 
represented by a projective object 
\[P_{\KZ} = \bigoplus_{ \la \in \vL^+} P(\la)^{\oplus \dim \KZ(L(\la))} \in \CO, \]
namely, we have 
$\KZ = \Hom_{\CO}({P_{\KZ}}, -)$. 
Moreover, 
by \cite[Theorem 5.15]{GGOR}, 
we have 
\[\He \cong (\End_{\CO} (P_{\KZ}))^{\opp}.\]  

By \cite[Theorem 5.16]{GGOR}, 
KZ functor 
is 
fully faithful 
when we restrict to projective objects in $\CO$. 
Thus, 
$\CO$ is a quasi-hereditary cover of $\He$.

Put 
$\A=\End_{\CO}(X)$, 
$\B=\He$ 
and 
$F=\KZ$, 
where $X$ is a progenerator of $\CO$ such that $X=P_{\KZ} \oplus P'$ for some projective object $P'$ in $\CO$. 
Then, 
these satisfy 
assumptions 
(A1),(A2),(A3) 
by \cite{GGOR}. 
Thus, 
all results in \S \ref{section quasi-hereditary}
hold for this setting. 
In particular, 
we put 
$S(\la) = \KZ (\D(\la))$ and 
$D(\la)= \KZ (L(\la))$ for 
$\la \in \vL^+$. 
Let 
$\vL^+_0=\{\la \in \vL^+ \,|\, D(\la) \not=0\}$, 
then 
$\{D(\la) \,|\, \la \in \vL^+_0\}$ 
gives a complete set of 
non-isomorphic simple $\He$-modules.

\para 
In the rest of this section, 
we recall a modular system and a decomposition map described in \cite{GGOR}. 
Let 
$\CC[ \wh{\Om}]$ 
be the polynomial ring 
over $\CC$ 
with indeterminates 
$\wh{\Om} = \{ \bk_{H,i} \,|\, H \in \CA/W,\, 1 \leq i \leq e_{H}-1 \}$.   
We have a homomorphism 
$\vf: \CC[\wh{\Om}] \ra \CC$ of $\CC$-algebras 
such that 
$\bk_{H,i} \mapsto k_{H,i}$.  
Put $\Fm=\Ker \vf$. 
Let 
$R$ 
be the completion of $\CC[ \wh{\Om}]$ at the maximal ideal $\Fm$. 
Then 
$R$ 
is a regular local ring with the unique maximal ideal 
$\wh{\Fm}=\big( (\bk_{H,i} - k_{H,i})_{H \in \CA/W,\,1 \leq i \leq e_H -1} \big)$. 
We have the canonical homomorphism 
$R \ra \CC$ such that $\bk_{H,i} \mapsto k_{H,i}$. 
Let $K$ be the quotient field of $R$. 

Let $\CH_R$ be the rational Cherednik algebra of $W$ over $R$ 
with parameters $\wh{\Om}$ (put $\bk_{H,0}=\bk_{H,e_H}=0$), 
and 
$\He_R$ be the cyclotomic Hecke algebra over $R$ associated to $\CH_R$. 
Then we have 
$\CH = \CC \otimes_R \CH_R$ and $\He = \CC \otimes_R \He_R$. 
Put 
$\CH_K = K \otimes_R \CH_K$ and $\He_K = K \otimes_R \He_R$. 
We denote objects over $X=R$ or $K$ by adding subscript $X$, 
e.g. $\CO_X$, $\D(\la)_X$,  $KZ_X$, $S(\la)_X$ $\cdots$. 

Under the modular system $(K,R,\CC)$, 
we can define the decomposition map 
\[ d_{K,\CC} : K_0(\He_K \cmod) \ra K_0 (\He \cmod)\] 
by 
$[M] \mapsto [\CC \otimes_R N]$, 
where $N$ is an $\He_R$-lattice of $M$.
Thanks to \cite[Theorem 5.13]{GGOR}, we have the following lemma. 

\begin{lem}
\label{lem decomposition map Specht}
For $\la \in \vL^+$, 
we have 
\[ d_{K,\CC}( [ S_K(\la)] ) =[S(\la)] .\]
\end{lem}

%%%%%%%%%%%%%%%%%%%%%%%%%%%%%%%%%%%%%%%%%%%%%%%%%%%%%%%%%%%%%%%%%%%%%%%%%%%%%%%%%%%%%%%%%%%%%%%%%%%%%%%%%%%%%%%%%%%%%

%%%%%%%%%%%%%%%%%%%%%%%%%%%%%%%%%%%%%%%%%%%%%%%%%%%%%%%%%%%%%%%%%%%%%%%%%%%%%%%%%%%%%%%%%%%%%%%%%%%%%%%%%%%%%%%%%%%%%

\section{Case of type $G(r,1,n)$}
\label{G(r,1,n)}

In this section, 
we consider the complex reflection group $W$ of type $G(r,1,n)$, i.e. $W=\FS_n \ltimes (\ZZ/r \ZZ)^n$. 
In this case, 
$\He$ is often called the Ariki-Koike algebra, 
and many results for representations of $\He$ are known by several authors. 

\para 
In this section, 
we use the modular system $(K,R,\CC)$ given in the previous section, 
and we take parameters as follows. 

Let 
$V$ be an $n$ dimensional vector space over $\CC$ 
with a basis $\{e_1,\cdots, e_n\}$. 
Then 
we have 
$W \subset \GL(V)$. 
Let 
$s_1, t_1 \in W$ be  reflections  
such that 
\begin{align}
\label{reflections}
s_1(e_k)= 
\begin{cases}
e_2 &\text{ if } k=1, 
\\
e_1 &\text{ if } k=2, 
\\
e_k &\text{ otherwise }, 
\end{cases}
\quad 
t_1(e_k)=
\begin{cases} 
\zeta e_1 &\text{ if } k=1, 
\\ 
e_k &\text{ otherwise }, 
\end{cases}\,\,
(\zeta= \exp(2 \pi \sqrt{-1}/r)), 
\end{align}
and 
$H_{s_1}$ (resp. $H_{t_1}$) 
be the reflecting hyperplane 
corresponding to 
$s_1$ (resp. $t_1$). 
Then 
$\{H_{s_1}, H_{t_1}\}$ 
gives a complete set of representatives of $W$-orbits of $\CA$, 
and 
we have 
$e_{H_{s_1}}=2$ and $e_{H_{t_1}}=r$. 
Hence, 
we can take parameters 
$\{h, k_1,\cdots,k_{r-1}\}$ 
(resp. $\{ \bh, \bk_1,\cdots, \bk_{r-1}\}$) 
of 
$\CH$ 
(resp. $\CH_X$ ($X=R \text{ or } K$)) 
such that 
$h=k_{H_{s_1},1}$ 
(resp. $\bh= \bk_{H_{s_1},1}$) 
and 
$k_j= k_{H_{t_1},j}$ 
(resp. $\bk_j=\bk_{H_{t_1},j}$) 
for $1 \leq j \leq r-1$.  
Then   
$\He$ (resp. $\He_R$, $\He_K$) 
is the associative algebra over $\CC$ (resp. $R$, $K$) 
defined by 
generators 
$T_0,T_1,\cdots,T_{n-1}$ 
with defining relations: 
\begin{align}
&(T_0 - 1)(T_0-Q_1)\cdots (T_0 - Q_{r-1})=0, 
\label{Hecke relation}\\
&(T_0 -1)(T_0+q)=0,  \notag \\
&T_0T_1T_0T_1=T_1T_0T_1T_0,  \notag \\
&T_i T_{i+1} T_i = T_{i+1} T_i T_{i+1} \quad (1\leq i \leq n-1), \notag \\
& T_i T_j =T_j T_i \quad (|i-j|>1), \notag
\end{align} 
where 
$Q_i=\exp (2 \pi \sqrt{-1}(k_i + \frac{i}{r}))$, 
$q=\exp (2 \pi \sqrt{-1} h)$ 
(resp. 
$Q_i=\exp (2 \pi \sqrt{-1}(\bk_i + \frac{i}{r}))$, 
$q=\exp (2 \pi \sqrt{-1} \bh)$).

\para 
Let 
\[ 
\CP_{n,r} = \left\{ \la=(\la^{(1)},\cdots,\la^{(r)}) \Biggm| 
	\begin{matrix} 
		\la^{(k)}=(\la_1^{(k)}, \la_2^{(k)}, \cdots) \in \ZZ_{\geq 0}^n 
			\text{ with } \la_1^{(k)} \geq \la_2^{(k)}\geq \cdots \\[1mm]
		\sum_{k=1}^r \sum_{i \geq 1} \la_{i}^{(k)} =n
	\end{matrix}
\right\} 
\]
be the set of $r$-partitions of size $n$. 
It is well-known that 
the isomorphism classes of simple $\CC W$-modules 
are indexed by $\CP_{n,r}$, 
thus we have 
$\vL^+=\CP_{n,r}$.

\para 
By \cite{DJM}, 
it is known that $\He_X$ ($X=K, R$ or $\CC$, we may omit the subscript $X$ when $X=\CC$) 
is a cellular algebra with respect to a poset $(\vL^+,\trianglerighteq)$, 
where \lq\lq $\trianglerighteq$" is the dominance order on $\vL^+$. 
We denote by 
$S^\la_X$ the Specht (cell) module for $\la \in \vL^+$ 
constructed by using the cellular basis in \cite{DJM}.

It is known that $\He_K$ is semi-simple, 
and 
$\{S^\la_K \,|\, \la \in \vL^+\}$ 
gives a complete set of non-isomorphic simple $\He_K$-modules.

By the general theory of cellular algebras (see \cite{GL} or \cite{M-book}), 
we can define the canonical bilinear form $\lan \, , \, \ran : S^\la \times S^\la \ra \CC$ 
by using the cellular basis. 
Put $\Rad S^\la = \{ x \in S^\la \,|\, \lan x,y \ran=0 \text{ for any } y \in S^\la\}$  
and $D^\la = S^\la / \Rad S^\la$. 
Let 
$\CK_{n,r}$ 
be the set of Kleshchev multi-partitions containing in $\vL^+$ 
(see e.g. \cite{A01} \cite{M} for the definition). 
Then it is known that 
$\{D^\la\,|\, \la \in \CK_{n,r}\}$ 
gives a complete set of non-isomorphic simple $\He$-modules by \cite{A01}.  

It is known that all composition factor of 
$S^\la$ 
belong to the same block of 
$\He$. 
Let \lq\lq $\sim$" be an equivalent relation on 
$\{S^\la \,|\, \la \in \vL^+\}$ 
defined in a similar way as the equivalent relation \lq\lq $\sim$" 
on $\{S(\la)\,|\, \la \in \vL^+\}$ 
in the previous section.   
Then it is known that 
\begin{align}
\label{S la mu block DJM}
S^\la \sim S^\mu \,\, \text{ if and only if } S^\la \text{ and } S^\mu \text{ belong to the same block of } \He.
\end{align}
By \eqref{S la mu block DJM}, 
we can classify the blocks of 
$\He$ 
by the equivalent classes of 
$\{S^\la\,|\, \la \in \vL^+\}$ 
with respect to \lq \lq $\sim$", 
and such equivalent classes are described by using some combinatorics in \cite{LM} as follows.  
For $\la \in \vL^+$, 
put 
\[[\la]=\{(i,j,k)\in \ZZ_{>0}^3 \,|\, 1 \leq j \leq \la^{(k)}_i, 1 \leq k \leq r\}. \]  
For $x=(i,j,k) \in [\la]$, 
we define 
\[ 
\res(x)= 
\begin{cases}
q^{j-i} Q_{k-1} & \text{if }q \not=1 \text{ and }Q_{k-1} \not=0, \\
(j-i, Q_{k-1}) & \text{if }q=1 \text{ and } Q_{l-1} \not=Q_{k-1} \text{ for } k \not=l, \\
Q_{k-1} &\text{otherwise}, 
\end{cases}
\]
where we put $Q_0=1$. 
Put  
$ \Res(\vL^+) = \{ \res(x) \,|\, x \in [\la] \text{ for some } \la \in \vL^+\}$. 
Then, 
we define an equivalent relation (called residue equivalence) \lq\lq $\sim_R$" on $\vL^+$ 
by 
\[
\la \sim_R \mu  \,\, \text{ if }\,\,  
\sharp \{ x \in [\la] \,|\, \res(x)=a\}= \sharp \{ y \in [\mu]\,|\, \res(y)=a\} 
\text{ for all } 
a \in \Res(\vL^+).
\]  

\begin{thm}[{\cite[Theorem 2.11]{LM}}] 
\label{thm LM}
For $\la,\mu \in \vL^+$, we have that 
\[ S^\la \sim S^\mu \,\, \text{ if and only if } \la \sim_R \mu.\]
\end{thm}

\para 
We take 
$\Irr W=\{E^\la \,|\, \la \in \vL^+\}$ 
such that 
$K \otimes_{\CC} E^\la \cong S^\la_K$ via the isomorphism $\He_K \cong K \otimes_{\CC} \CC W$. 
Since 
$S_K^\la = K \otimes_R S_R^\la$ and $S^\la = \CC \otimes_R S_R^\la$, 
we have that 
\begin{align}
\label{decomposition map DJM Specht} 
d_{K,\CC}([S_K^\la]) = [S^\la]. 
\end{align}
It is also well-known that 
$K \otimes_{\CC} E^\la \cong S_K(\la)$ via the isomorphism $\He_K \cong K \otimes_{\CC} \CC W$ 
(see before \cite[Theorem 5.13]{GGOR}).  
Thus, we have $S_K^\la \cong S_K(\la)$ as $\He_K$-modules. 
Then 
Lemma \ref{lem decomposition map Specht} together with \eqref{decomposition map DJM Specht} 
implies that 
\begin{align}
\label{iso in Grothen S la}
[S(\la)] = [S^\la] \quad \text{ in } K_0(\He \cmod) \text{ for } \la \in \vL^+.
\end{align}
Note that 
$S(\la) \not\cong S^\la$ as $\He$-modules in general. 
Hence, 
$\Top S(\la) \not\cong \Top S^\la$ in general. 
Moreover, 
$\vL^+_0 \not= \CK_{n,r}$ in general. 
Let 
\[\theta  : \vL^+_0 \ra \CK_{n,r}\] 
be the bijection 
such that 
$D(\la) \cong D^{\theta(\la)}$ as $\He$-modules. 
Then we have the following proposition.

\begin{prop}
For $\la \in \vL^+$ and $\mu \in \vL^+_0$, 
we have 
\[[\D(\la) : L(\mu)]_{\CO} = [S(\la) : D(\mu)]_{\He} = [S^\la : D^{\theta(\mu)}]_{\He}.\] 

\end{prop}

\begin{proof}
The first equality is clear since the KZ functor is exact. 
By \eqref{iso in Grothen S la}, 
we have $[S(\la)] = [S^\la]$ in $K_0(\He \cmod)$, 
and $D(\mu) \cong D^{\theta(\mu)}$. 
Thus, we have the second equality. 
\end{proof}

The following theorem gives a relation between blocks of $\CO$ and blocks of $\He$.  
In particular, 
we obtain the classification of blocks of $\CO$ 
by using the residue equivalence.  

\begin{thm}
\label{thm G(r,1,n)}
For $\la, \mu \in \vL^+$, we have 
\[ \D(\la) \sim \D(\mu) 
\Leftrightarrow  
S(\la) \sim S(\mu) 
\Leftrightarrow 
S^\la \sim S^\mu 
\Leftrightarrow 
\la \sim_R \mu.
\] 
\end{thm} 

\begin{proof}
The first equivalence is Proposition \ref{prop equivalent D la sim D mu S la sim S mu}. 
The second equivalence follows from \eqref{iso in Grothen S la}. 
The third equivalence is Theorem \ref{thm LM}.
\end{proof}

\remark 
By \cite{R}, 
under a certain condition for parameters, 
it is known that 
$\CO$ is equivalent to $\Sc_{n,r} \cmod$ 
as highest weight categories,  
where $\Sc_{n,r}$ is the cyclotomic $q$-Schur algebra associated to $\He$ defined in \cite{DJM}.  
In this case, 
we have $S(\la) \cong S^\la$, 
and  $\theta$ is the identity map (in particular, $\vL^+_0 = \CK_{n,r}$).  
So, the above theorem is  known  by \cite{LM}. 
However, 
the above theorem claim that 
a classification of blocks of $\CO$ is given by the equivalent relation \lq\lq $\sim_R$" on $\vL^+$ 
(residue equivalence) 
even if the case where 
$\CO$ is not equivalent to $\Sc_{n,r} \cmod$.

%%%%%%%%%%%%%%%%%%%%%%%%%%%%%%%%%%%%%%%%%%%%%%%%%%%%%%%%%%%%%%%%%%%%%%%%%%%%%%%%%%%%%%%%%%%%%%%%%%%%%%%%%%%%%%%%%%%%%

%%%%%%%%%%%%%%%%%%%%%%%%%%%%%%%%%%%%%%%%%%%%%%%%%%%%%%%%%%%%%%%%%%%%%%%%%%%%%%%%%%%%%%%%%%%%%%%%%%%%%%%%%%%%%%%%%%%%%

\section{Case of type $G(r,p,n)$} 
In this section, 
we consider the case where  
$W$ 
is the complex reflection group of type 
$G(r,p,n)$, 
where 
$r=pd$ for some $d\geq 1$. 
It is well-known that 
the complex reflection group of type 
$G(r,p,n)$ 
is a normal subgroup of the complex reflection group of type 
$G(r,1,n)$
with the index 
$p$, 
and we will study some relations between type 
$G(r,1,n)$ 
and type 
$G(r,p,n)$. 
Hence, 
we denote by 
$W^{\dag}$ 
the complex reflection group of type 
$G(r,1,n)$,   
and 
we use the results in the previous section for 
$W^\dag$.  
In this section, 
we use the notations in 
\S \ref{RCA} 
for corresponding objects of type 
$G(r,p,n)$, 
and 
we denote by adding the superscript $\dag$ 
for corresponding objects of type 
$G(r,1,n)$, 
e.g. $\CH^\dag$, $\He^\dag$, $\D^\dag (\la)$, $KZ^\dag$, $S^\dag(\la)$, $\cdots$. 
Let 
$\Irr W^\dag =\{ E^{\dag \la} \,|\, \la \in \CP_{n,r}\}$ 
be a complete set of non-isomorphic simple $\CC W^\dag$-modules 
considered in the previous section.

\para 
Let 
$V$ be an $n$ dimensional vector space over $\CC$ 
with a basis $\{e_1,\cdots, e_n\}$. 
Then 
we have 
$W \subset \GL(V)$. 
Recall that 
$s_1, t_1 \in \GL(V)$ 
is a reflection 
defined in 
\eqref{reflections}. 
Then 
$s_1$ (resp. $t_1^p$ in the case where $p\not=r$)
is a reflection contained in $W$, 
and 
let 
$H_{s_1}$ (resp. $H_{t^p_1}$)  
be the reflecting hyperplane 
corresponding to 
$s_1$ (resp. $t^p_1$). 
In the case where $p \not=r$, 
$\{H_{s_1}, H_{t_1^p}\}$ 
gives a complete set of representatives of $W$-orbits of $\CA$, 
and 
we have 
$e_{H_{s_1}}=2$ and $e_{H_{t_1}^p}=d$. 
Hence, 
we can take parameters 
$\{h, k_1,\cdots,k_{d-1}\}$ 
(resp. $\{ \bh, \bk_1,\cdots, \bk_{d-1}\}$) 
of 
$\CH$ 
(resp. $\CH_X$ ($X=R \text{ or } K$)) 
such that 
$h=k_{H_{s_1},1}$ 
(resp. $\bh= \bk_{H_{s_1},1}$) 
and 
$k_j= k_{H_{t^p_1},j}$ 
(resp. $\bk_j=\bk_{H_{t^p_1},j}$) 
for $1 \leq j \leq d-1$.  
On the other hand, 
in the case where $r=p$, 
$\CA$ is the $W$-orbit of $\CA$ itself. 
Hence 
$\CH$ 
(resp. $\CH_X$ ($X=R \text{ or } K$)) 
has a parameter 
$\{h\}$ (resp. $\{\bh \}$).

Then   
$\He$ (resp. $\He_R$, $\He_K$) 
is the associative algebra over $\CC$ (resp. $R$, $K$) 
defined by 
generators 
$a_0,a_1', a_1,a_2,\cdots,a_{n-1}$ 
with defining relations: 
\begin{align*}
&(a_0 - 1)(a_0-x_1)\cdots (a_0 - x_{d-1})=0, 
\\
&(a_1' -1)(a_1'+q)=0,  
\quad 
(a_i -1)(a_i +q)=0 \quad (1 \leq i \leq n-1), 
\\
&a_0 a_1' a_1 = a_1' a_1 a_0, 
\quad 
a_1' a_2 a_1' = a_2 a_1' a_2, 
\quad 
(a_2 a_1' a_1)^2 = (a_1' a_1 a_2)^2,
\\
&a_0 a_i= a_i a_0 \quad (2 \leq i \leq n-1), 
\quad 
a_1' a_j =a_j a_1' \quad (3 \leq j \leq n-1), 
\\
&\underbrace{a_1 a_0 a_1' a_1 a_1' a_1 a_1' \cdots}_{p+1 \text{ factors }} 
	= \underbrace{a_0 a_1' a_0 a_1' a_0 a_1' a_0 \cdots}_{p+1 \text{factors }},
\\
&a_i a_{i+1} a_i = a_{i+1} a_i a_{i+1} \quad (1 \leq i \leq n-2), 
\\
&a_i a_j = a_j a_i \quad (1 \leq i < j-1 \leq n-2), 
\end{align*} 
where 
$x_i=\exp (2 \pi \sqrt{-1}(k_i + \frac{i}{d})$, 
$q=\exp (2 \pi \sqrt{-1} h)$ 
(resp. 
$x_i=\exp (2 \pi \sqrt{-1}(\bk_i + \frac{i}{d})$, 
$q=\exp (2 \pi \sqrt{-1} \bh)$) 
(see \cite{BMR} or \cite{A95} for braid relations).

\para 
Put 
\begin{align*}
&k_{c \cdot p +j}^\dag = k_c + \frac{c}{d} +\frac{j}{p} - \frac{c \cdot p +j}{r} 
\quad 
(0 \leq c \leq d-1,\,  0 \leq j \leq p-1), 
\\
&\bk_{c \cdot p +j}^\dag = \bk_c + \frac{c}{d} +\frac{j}{p} - \frac{c \cdot p +j}{r} 
\quad 
(0 \leq c \leq d-1,\,  0 \leq j \leq p-1), 
\end{align*}
where we set $k_0=\bk_0=0$.

Throughout this section, 
let 
$\CH^\dag$ (resp. $\CH^\dag_X$ ($X=R \text{ or } K$))  
be the ratinal Cherednik algebra 
associated to $W^\dag$ 
with parameters 
$\{h, k_1^\dag,\cdots,k_{r-1}^\dag\}$ 
(resp. $\{ \bh, \bk_1^\dag,\cdots, \bk_{r-1}^\dag\}$) 
such that 
$h=k^\dag_{H_{s_1},1}$ 
(resp. $\bh= \bk^\dag_{H_{s_1},1}$) 
and 
$k_j^\dag= k^\dag_{H_{t_1},j}$ 
(resp. $\bk_j^\dag=\bk^\dag_{H_{t_1},j}$) 
for $1 \leq j \leq r-1$.  
Since 
\begin{align*}
\exp \left( 2 \pi \sqrt{-1} \big(k_{c \cdot p +j}^\dag + \frac{c\cdot p +j}{r}\big) \right)
&= 
\exp \left( 2 \pi \sqrt{-1} \big(  k_c + \frac{c}{d} +\frac{j}{p} \big) \right)
\\[2mm]
&= x_c \, \xi^{j} \quad (\xi=\exp ( 2 \pi \sqrt{-1} / p)),
\end{align*}
where we put $x_0=1$ 
(similar for $\bk^\dag_{c \cdot p+j}$), 
the defining relation 
\eqref{Hecke relation} 
of $\He^\dag$ (resp. $\He^\dag_X$) 
replaced by 
\begin{align*}
(T_0^p-1)(T_0^p -x_1^p) \cdots (T_0^p - x_{d-1}^p)=0.
\end{align*}
Since 
$\He_K^\dag$ is semi-simple (thus, $\CO_K^\dag$ is also semi-simple) 
by \cite{Ari94} , 
we can obtain any results for type $G(r,1,n)$ in the previous sections  
even if the case of these parameters.

By \cite[Proposition 1.6]{A95}, 
there is the injective algebra homomorphism 
$\vf : \He_X \ra \He_X^\dag$ ($X=\CC,\, R \text{ or }K$) 
such that 
$\vf(a_0)=T_0^p$, $\vf(a'_1)=T_0^{-1} T_1 T_0$, $\vf(a_i)=T_i$ ($1 \leq i \leq n-1$). 
Under this injective homomorphism $\vf$, 
we regard $\He_X$ as a subalgebra of $\He_X^\dag$.

\para 
\label{par res D to B}
For $M^\dag \in \CC W^\dag \cmod$, 
we denote by 
$M^\dag \dwar$ 
the restriction of the action to $\CC W$.   
For $\la=(\la^{(1)},\cdots, \la^{(r)}) \in \CP_{n,r}$ 
and 
$i \in \ZZ$, 
we define 
$\la[i]=(\la[i]^{(1)}, \cdots, \la[i]^{(r)})
\in \CP_{n,r}$ 
by 
\[
\la[i]^{(c\cdot p + j)}= 
\la^{(c \cdot p +k)}, 
\quad 
(0 \leq c \leq d-1, \,\,  1 \leq j \leq p), 
\]
where 
$c \cdot p < c \cdot p + k \leq (c+1) \cdot p$  
such that 
$k \equiv j+i \mod p$.
For an example, 
if $r=6$ and $p=3$, we have 
\begin{align*}
& \la[1]=(\la^{(2)}, \la^{(3)}, \la^{(1)}, \vdots \, \la^{(5)}, \la^{(6)}, \la^{(4)}), 
\\
& \la[2]=(\la^{(3)}, \la^{(1)}, \la^{(2)}, \vdots \, \la^{(6)}, \la^{(4)}, \la^{(5)}), 
\\
& \la[3]=(\la^{(1)}, \la^{(2)}, \la^{(3)}, \vdots \, \la^{(4)}, \la^{(5)}, \la^{(6)})=\la. 
\end{align*} 
Let 
$\Fk_\la$ be the minimum positive integer such that 
$\la[\Fk_\la]=\la$. 
It is clear that 
$\Fk_\la \mid p$.  
Put 
$\Fd_\la = p/\Fk_\la$. 
Then 
we have that 
$\la[ i + \Fk_\la] =\la[i]$.
Let 
$\sim_{\ast}$ 
be the equivalent relation on $\CP_{n,r} \times \ZZ / p \ZZ$ 
defined by 
\[ (\la, \ol{j}) \sim_\ast (\la[i], \, \ol{c \cdot \Fd_\la + j}) 
\quad (i,\, c \in \ZZ),
\] 
where we denote by 
$\ol{m}$ 
the image of 
$m \in \ZZ$ 
in 
$\ZZ / p \ZZ$. 
Let 
$\vL^+$ be the set of equivalent classes 
of $\CP_{n,r} \times \ZZ / p\ZZ$ 
with respect to the relation $\sim_\ast$, 
and 
we denote by 
$\la \lan j \ran \in \vL^+$ 
the equivalent class 
containing 
$(\la,\ol{j}) \in \CP_{n,r} \times \ZZ / p \ZZ$. 
Thus we have that 
$\la \lan j \ran = \la[i] \lan j \ran 
= \la \lan c \cdot \Fd_\la +j \ran = \la[i] \lan c \cdot \Fd_\la +j \ran$ 
for $i, \, c \in \ZZ$.  
Then it is known that, 
\begin{align}
\label{decom E G(r,1,n) to G(r,p,n)}
E^{\dag \la} \dwar \, 
	\cong E^{\dag \la[i]} \dwar \, 
	\cong E^{\la\lan 1 \ran} \oplus \cdots \oplus E^{\la\lan \Fd_\la \ran} \,\,\, (i \in \ZZ)  
\quad \text{ as $\CC W$-modules}, 
\end{align}
for some simple $\CC W$-modules $E^{\la \lan j \ran}$ ($1 \leq j \leq \Fd_\la$), 
and  
$
\{ E^{\la\lan j \ran} \bigm| \la \lan j \ran \in \vL^+ \}
$
gives a complete set of pairwise non-isomorphic simple $\CC W$-modules. 
Hence, 
we have that 
\[
\Irr \CC W = \{ E^{\la \lan j \ran} \,|\, \la \lan j \ran \in \vL^+ \} .
\] 
Moreover, 
we have that 
\begin{align}
\label{induce G(r,p,n) to G(r,1,n)}
E^{\la \lan j \ran} \upar \, \cong E^{\dag \la[1]} \oplus \cdots \oplus E^{\dag \la[\Fk_\la]}
\quad \text{as $W^\dag$-modules} \,\, (1 \leq j \leq \Fd_\la).
\end{align}

\para 
For $M^\dag \in \He_X^\dag \cmod$, 
we denote by 
$M^\dag \dwar$ 
the restriction of the action to $\He_X$. 
On the other hand, 
for $N \in \He_X \cmod$, 
we denote by 
$N \upar $ 
the induced module $\He^\dag_X \otimes_{\He_X} N$. 

Then, 
by \eqref{decom E G(r,1,n) to G(r,p,n)}, 
we have that
\begin{align}
\label{S K decom G(r,1,n) to G(r,p,n)}
&S_K^\dag (\la) \dwar \,
\cong 
S_K(\la \lan 1 \ran) \oplus \cdots \oplus S_K(\la \lan \Fd_\la \ran ) 
\quad \text{for } \la \in \CP_{n,r} 
\end{align}
and, by \eqref{induce G(r,p,n) to G(r,1,n)},  we have that, 

\begin{align}
\label{S K induce G(r,p,n) to G(r,1,n)}
S_K(\la \lan j \ran ) \upar \, 
\cong 
S_K^\dag (\la[1]) \oplus \cdots \oplus S_K^\dag (\la[\Fk_\la]) 
\quad \text{for } \la\lan j \ran \in \vL^+. 
\end{align}

We define the group homomorphism 
$\Res_X : K_0(\He_X^\dag \cmod) \ra K_0(\He_X \cmod)$ 
by $[M^\dag] \mapsto [M^\dag \dwar]$. 
We also define the group homomorphism 
$\Ind_X : K_0(\He_X \cmod) \ra K_0(\He_X^\dag \cmod)$ 
by $[N] \mapsto [N\upar]$. 
Since $\He_X^\dag$ is a free right $\He_X$-module, 
induced functor from $\He_X \cmod $ to $\He_X^\dag \cmod$ 
is exact. 
Thus $\Ind_X$ is well-defined. 
Then we have the following lemma.

\begin{lem}\
\label{lem G(r,1,n) G(r,p,n) Specht}
\begin{enumerate}
\item 
For $\la \in \CP_{n,r}$, 
we have 
\[ [
S^\dag(\la) \dwar] 
= 	
[S(\la \lan 1 \ran) ] + \cdots + [S(\la \lan \Fd_\la \ran)]
\quad \text{in $K_0(\He \cmod)$.}
\]

\item 
For $\la \lan j \ran \in \vL^+$, 
we have 
\[ 
[S(\la\lan j \ran ) \upar] 
= 
[S^\dag (\la[1])] + \cdots + [ S^\dag (\la[\Fk_\la])] 
\quad \text{in $K_0(\He^\dag \cmod)$.}
\]	

\end{enumerate}

\end{lem}
\begin{proof}
(\roi) 
By Lemma \ref{lem decomposition map Specht} and \eqref{S K decom G(r,1,n) to G(r,p,n)}, 
we have 
\begin{align*}
 d_{K,\CC}( [ S_K^\dag(\la)\dwar]) 
&= 
 d_{K, \CC}\big( [S_K(\la \lan 1 \ran)] + \cdots + [ S_K (\la \lan \Fd_\la \ran) ] \big) 
\\
&=
 [S(\la \lan 1 \ran) ] + \cdots + [ S (\la \lan \Fd_\la \ran) ]. 
\end{align*}
On the other hand, 
by the definition of decomposition maps, 
we have 
\[ d_{K,\CC} ( [ S_K^\dag (\la)\dwar ]) = [S^\dag(\la)\dwar]. \]
Then (\roi) was proven. 
By using \eqref{S K induce G(r,p,n) to G(r,1,n)} together with Lemma \ref{lem decomposition map Specht}, 
we have (\roii) in a similar way as in (\roi). 
\end{proof}

\para 
\label{simple HD}
We recall some relations 
between  simple $\He$-modules 
and simple $\He^\dag$-modules 
which have been studied in 
\cite{GJ} and \cite{Hu09}

Let 
$\{S^{\dag \la} \,|\, \la \in \CP_{n,r}\}$ 
be the set of Specht modules of $\He^\dag$ constructed in \cite{DJM} 
as seen in the previous section. 
Then  
$\{D^{\dag \la} \,|\, \la \in \CK_{n,r}\}$ 
is a complete set of simple $\He^\dag$-modules. 

Let 
$\s$ 
be the algebra automorphism of 
$\He^\dag$ 
defined by 
$\s(T_0)= \xi T_0$ 
($\xi=\exp(2 \pi \sqrt{-1}/ p)$), 
$\s(T_i)= T_i$ 
for $i=1,\cdots,n-1$. 
Then we see that 
the restriction $\s|_{\He}$ of $\s$ to $\He$ 
is the identity map on $\He$.  
We also define 
the algebra automorphism $\t$ of $\He^\dag$ 
by $\t(x) = T_0^{-1} x T_0$ for $x \in \He^\dag$. 
Then we have that $\t (\He)= \He$.

For 
$M^\dag \in \He^\dag \cmod$,  
let 
$(M^\dag)^\s$ 
be the twisted 
$\He^\dag$-module of $M$ 
via 
$\s$. 
Since 
$\s|_{\He}$ is identity map, 
we have that 
$(M^\dag)^\s \dwar \cong M^\dag \dwar$ 
as $\He$-modules.  
Similarly, 
for 
$N \in \He \cmod$,  
let 
$N^\t$ 
be the twisted 
$\He$-module of $N$ 
via   
$\t$. 
For $\la \in \CK_{n,r}$ and $i \in \ZZ$, 
we define 
$\la[i]^\flat $ 
by 
$(D^{\dag \la})^{\s^i} \cong D^{\dag \la[i]^\flat}$. 
Let 
$\Fk_\la^\flat$ 
be the minimum positive integer 
such that 
$\la[\Fk_\la^\flat]^\flat = \la$ 
(thus $(D^{\dag \la})^{\s^{\Fk_\la^\flat}} \cong D^{\dag \la}$), 
and put 
$\Fd_\la^\flat = p/ \Fk_\la^\flat$.  
Let 
$D$ 
be a simple $\He$-submodule of $D^{\dag \la} \dwar$. 
Then by \cite[Lemma 2.2]{GJ}, 
$\Fd_{\la}^\flat$ is the minimum positive integer such that 
$D^{\t^{\Fd_\la^\flat}} \cong D$. 
Moreover we have that, 
for $\la \in \CK_{n,r}$ and $i=1,\cdots, \Fk_\la^\flat$,  
\begin{align}
\label{decom D G(r,1,n) to G(r,p,n)} 
D^{\dag \la} \dwar \, 
\cong 
D^{\dag \la[i]^\flat} \dwar 
\cong D \oplus D^{\t} 
	\oplus \cdots \oplus D^{\t^{\Fd_{\la}^\flat -1}} 
\quad 
\text{ as $\He$-modules.}
\end{align} 

Let 
$\sim_\star $ 
be the equivalent relation on $\CK_{n,r} \times \ZZ/ p \ZZ$ 
defined by 
\[ 
(\la, \ol{j}) \sim_\star (\la[i]^\flat, \ol{ c \cdot \Fd_\la^\flat +j}) 
\quad 
(i, \, c \in \ZZ).
\]
We denote by $(\CK_{n,r} \times \ZZ / p \ZZ)/_{\sim_\star}$ 
the set of equivalent classes of $\CK_{n,r} \times \ZZ / p \ZZ$   
with respect to the relation $\sim_\star$, 
and 
we denote by 
$\la \lan j \ran^\flat \in (\CK_{n,r} \times \ZZ / p \ZZ)/_{\sim_\star}$ 
the equivalent class containing 
$(\la, \ol{j}) \in \CK_{n,r} \times \ZZ / p\ZZ$.  
Then, by \cite[Lemma 2.2]{GJ} (see also \cite[Propsotion 2.4]{Hu09}), 
\[ \big\{ D^{\la\lan j \ran^\flat} \bigm| \la \lan j \ran^\flat \in  
 (\CK_{n,r} \times \ZZ / p \ZZ)/_{\sim_\star} \big\}
\]
gives a complete set of pairwise non-isomorphic simple $\He$-modules, 
where 
we put $D^{\la \lan j \ran^\flat} = D^{\t^{j}}$ 
for some simple 
$\He$-submodule $D$ of $D^{\dag \la} \dwar$ 
(see \eqref{decom D G(r,1,n) to G(r,p,n)}).

By \cite[Lemma 2.2]{GJ}, 
we also have that, 
for $\la \lan j \ran^\flat \in (\CK_{n,r} \times \ZZ / p \ZZ)/_{\sim_\star}$,  
\begin{align}
\label{induce D G(r,p,n) to G(r,1,n)} 
D^{\la \lan j \ran^\flat} \upar \, 
\cong 
D^{\dag \la[1]^\flat} \oplus \cdots \oplus D^{\dag \la[\Fk_\la^\flat]^\flat} 
\quad 
\text{ as $\He^\dag$-modules.}   
\end{align}

\remarks 
\label{remark involotion tau}
(\roi) 
For $\la \in \CK_{n,r}$, 
$\la[i]^\flat$ ($1 \leq i \leq \Fk_\la^\flat$) 
is described in \cite{Hu03} (the case of type D), \cite{Hu07} (the case of type $G(r,r,n)$) 
and \cite{GJ}, \cite{Hu09} (general case).

(\roii) 
Recall that 
$\{D(\la\lan i \ran') \,|\, \la\lan i \ran' \in \vL^+_0 \}$ 
gives a complete set of non-isomorphic simple $\He$-modules 
(Lemma \ref{lem projective simple B}). 
Hence, 
there exists the bijection 
$\eta : \vL^+_0 \ra (\CK_{n,r} \times \ZZ / p \ZZ)/_{\sim_\star}$ 
such that 
$D(\la\lan i\ran') \cong D^{\eta(\la\lan i \ran')}$. 
\\

Now we have the following proposition.

\begin{prop}
\label{prop decom rel G(r,1,n) G(r,p,n)}
For $\la \in \CP_{n,r}$ and $\mu \in \CK_{n,r}$, we have the following. 
\end{prop}
\begin{enumerate}
\item 
$\dis 
\sum_{s=1}^{\Fd_\la} \, [S(\la \lan s \ran) : D^{\mu \lan i \ran^\flat}]_{\He} 
= 
\sum_{t=1}^{\Fk_\mu^\flat} \, [S^{\dag \la[j]} : D^{\dag \mu[t]^\flat}]_{\He^\dag}
$ 
\quad 
$(1 \leq i \leq \Fd_\mu^\flat$, $1 \leq j \leq \Fk_\la)$.

\item 
$\dis 
\sum_{s=1}^{\Fd_\mu^\flat} \, [S(\la \lan i \ran) : D^{\mu \lan s \ran^\flat}]_{\He} 
= 
\sum_{t=1}^{\Fk_\la} [S^{\dag \la[t]} : D^{\dag \mu[j]^\flat}]_{\He^\dag}
$ 
\quad 
$(1 \leq i \leq \Fd_\la $, $1 \leq j \leq \Fk_\mu^\flat)$.
\end{enumerate}
\begin{proof}
Let 
\[ S^{\dag \la[j]} = M_k \supset M_{k-1} \supset \cdots \supset M_1 \supset M_0 =0 \]
be a composition series of $S^{\dag \la[j]}$ in $\He^\dag \cmod$ 
such that 
$M_l/ M_{l-1} \cong D^{\dag \mu_l}$. 
Applying the restriction functor, 
we have the filtration 
\[S^{\dag \la[j]}\dwar \, 
= M_k \dwar \, \supset  M_{k-1}\dwar \, \supset \cdots \supset M_1\dwar \, \supset M_0 \dwar \, =0 \]
such that 
$M_l \dwar\, / M_{l-1}\dwar \, \cong D^{\dag \mu_l}\dwar$ in $\He \cmod$. 
Thus, 
by \eqref{decom D G(r,1,n) to G(r,p,n)}, 
we have 
\begin{align} 
\label{decom rel 1}  
	[ S^{\dag \la[j]} \dwar \, : D^{\mu \lan i\ran^\flat} ]_{\He} 
= 
	\sum_{t=1}^{\Fk_\mu^\flat}
	[ S^{\dag \la[j]} : D^{\dag \mu[t]^\flat}]_{\He^\dag}.
\end{align}

On the other hand, 	
by \eqref{iso in Grothen S la} and Lemma \ref{lem G(r,1,n) G(r,p,n) Specht} (\roi) 
together with $S^{\dag \la} \dwar \, \cong S^{\dag \la[j]} \dwar$, 
\begin{align}
\label{decom rel 2}
	[ S^{\dag \la[j]} \dwar \, : D^{\mu \lan i\ran^\flat} ]_{\He} 
=   \sum_{s=1}^{\Fd_\la} \, [S(\la \lan s \ran) : D^{\mu \lan i \ran^\flat}]_{\He}. 
\end{align}
\eqref{decom rel 1} and \eqref{decom rel 2} imply (\roi).
Next we prove (\roii). 
Let 
\[ S(\la\lan i \ran ) = N_k \supset N_{k-1} \supset \cdots \supset N_1 \supset N_0 =0 \]
be a composition series of $S(\la \lan i \ran)$ in $\He \cmod$ 
such that 
$N_l/ N_{l-1} \cong D^{\mu_l \lan j_l \ran^\flat}$. 
Applying the induced functor, 
we have the filtration 
\[
	S(\la \lan i \ran)\upar \, 
	= N_k \upar\, \supset  N_{k-1}\upar\, \supset \cdots \supset N_1\upar \supset N_0 \upar =0 
\]
such that 
$N_l \upar\, / N_{l-1}\upar \, \cong D^{\mu_l \lan j_l \ran^\flat}\upar $ in $\He^\dag \cmod$. 
Thus,
by \eqref{induce D G(r,p,n) to G(r,1,n)}, 
we have 
\begin{align}
\label{decom rel 3}
 \big[ S(\la\lan i\ran ) \upar \, : D^{\dag \mu[j]^\flat} \big]_{\He^{\dag}} 
= \sum_{s=1}^{\Fd_\mu^\flat} \, [S(\la \lan i \ran) : D^{\mu \lan s \ran^\flat}]_{\He}.
\end{align}

On the other hand, 	
by \eqref{iso in Grothen S la} and Lemma \ref{lem G(r,1,n) G(r,p,n) Specht} (\roii), 
we have 
\begin{align}
\label{decom rel 4}
\big[ S(\la\lan i\ran ) \upar \, : D^{\dag \mu[j]^\flat} \big]_{\He^{\dag}}  
= 
\sum_{t=1}^{\Fk_\la} [S^{\dag \la[t]} : D^{\dag \mu[j]^\flat}]_{\He^\dag}.
\end{align}
\eqref{decom rel 3} and \eqref{decom rel 4} 
imply (\roii). 
\end{proof}

\para 
Recall that 
\lq\lq $\sim_R$" 
is the residue equivalence on $\CP_{n,r}$ defined in the previous section. 
We define an equivalent relation \lq\lq $\approx $" on $\CP_{n,r}$ 
by 
$\la \approx \mu$ 
if $\la \sim_R \mu [j]$ 
for some $j \in \ZZ$. 
Put 
\[ 
\vG = \big\{ \la \in \CP_{n,r} \bigm|
\la \not\sim_R   \mu 
	\text{ for any } \mu \in \CP_{n,r} \text{ such that } \mu \not= \la \big\}.
\]
We see easily that 
$\la \sim_R \mu$ 
if and only if 
$\la[i] \sim_R \mu[i]$ 
for any $i \in \ZZ$. 
Thus, 
we have that 
$\la[i] \in \vG$ if $\la \in \vG$. 
Then we have the following proposition.

\begin{prop}
\label{prop sim sim sim}
For $\la \in \CP_{n,r} \setminus \vG$, we have that 
\[ 
S(\la \lan 1 \ran) \sim S(\la \lan 2 \ran) \sim \cdots \sim S(\la \lan \Fd_\la \ran).
\]  
\end{prop}

\begin{proof}
If $\Fk_\la=p$, 
there is nothing to prove 
since $\Fd_\la =1$. 
Hence, we assume that $\Fk_\la \not=p$. 
First, we show the following claim. 
\begin{description}
\item[(claim)] 
For $\la \in \CP_{n,r} \setminus \vG$ 
such that $\Fk_\la \not=p$,   
we can take $\mu \in \CP_{n,r}$ 
such that 
$\la \sim_R \mu$, 
and that 
$\Fk_\mu=p$ (thus $\Fd_\mu =1$).  
\end{description}  

Since 
$\la \in \CP_{n,r} \setminus \vG$, 
we can take 
$\mu \in \CP_{n,r}$ 
such that 
$\la \sim_R \mu$ 
and 
$\mu \not=\la$. 
By \cite[Theorem 2.11]{LM}, 
it is known that 
$\la \sim_R \mu$ 
if and only if 
$\la \sim_J \mu$, 
where 
\lq\lq $\sim_J$" 
is the Jantzen equivalence on $\CP_{n,r}$ 
(see \cite[Definition 2.8]{LM} for deinitions).
By the definition of the Jantzen equivalence, 
we may assume that 
$\mu$ 
obtained by  unwrapping a rim hook 
$r_x^\la$ 
from $\la$,  
and wrapping another rim hook 
$r_y^\mu$ from $[\la] \setminus r_x^\la$. 
Namely, 
we have 
$[\la] \setminus r_x^\la = [\mu] \setminus r_y^\mu$ 
(See \cite{LM} for notations here). 
Suppose that 
$x \in \la^{(i)}$ and $y \in \mu^{(j)}$. 
Then 
$[\la] \setminus r_x^\la = [\mu] \setminus r_y^\mu$ 
implies that 
\begin{align}
\label{la i la j mu i mu j} 
\la^{(i)} \not= \mu^{(i)}, \,\,   
\la^{(j)} \not= \mu^{(j)}  
\text{ and } 
\la^{(l)}=\mu^{(l)} \text{ for } l \not=i,j.
\end{align} 
Note that 
$\mu^{(i)} \not= \mu^{(j)}$ if $\la^{(i)}=\la^{(j)}$ and $i \not=j$. 
Thus, 
we have that 
$\mu^{(i)} \not= \mu^{(l)}$ 
for any $l\not=i$  
such that 
$l \equiv i \mod \Fk_\la$ 
and 
$c \cdot p < l \leq (c+1)\cdot p$ 
when 
$c \cdot p < i \leq (c+1)\cdot p$. 
This implies that 
\begin{align}
\label{k la not divide k mu}
\Fk_\la \nmid \Fk_\mu 
\text{ unless } \Fk_\mu =p.
\end{align} 

In the case where $p$ is a prime number,  
\eqref{k la not divide k mu} implies 
$\Fk_\mu =p$ 
since 
$\Fk_\la=1$ by $\Fk_\la \mid p$ and $\Fk_\la \not=p$. 
In the case where  $p=4$, 
one can easily check that 
$\Fk_\mu=p$ directly. 
Let $p \geq 6$ be not a prime number. 
Assume that  $\Fk_\mu \not=p$. 
Then we have $\Fk_\la \nmid \Fk_\mu$ by \eqref{k la not divide k mu}. 
In a similar way as in the above arguments, 
we have $\Fk_\mu \nmid \Fk_\la$ (note that $\Fk_\la \not=p$). 
By the conditions 
$p \geq 6$, 
$\Fk_\la \nmid \Fk_\mu$ 
and 
$\Fk_\mu \nmid \Fk_\la$, 
one sees that 
there are at least three integers $x_1, x_2, x_3$ 
such that 
$\la^{(x_l)} \not= \mu^{(x_l)}$ ($l=1,2,3$). 
However, 
this contradicts to \eqref{la i la j mu i mu j}. 
Thus we have $\Fk_\mu =p$, 
and the claim was proved.

Thanks to the claim, 
we can take 
$\mu \in \CP_{n,r}$ such that 
$\la \sim_R \mu$, and that $\Fd_\mu =1$. 
Then we can take a sequence 
$\la = \la_0, \cdots, \la_{k}= \mu$ 
satisfying the following two conditions: 
\begin{itemize}
\item 
$S^{\dag \la_{i-1}}$ and $S^{\dag \la_{i}}$ have a common composition factor $D^{\dag \nu_i}$. 
\item 
There exists an integer 
$l$ such that $\Fd_{\la_i} \not=1$ for any $i< l$, 
and that $\Fd_{\la_l}=1$.   
\end{itemize}
By Proposition \ref{prop decom rel G(r,1,n) G(r,p,n)} (\roi), 
one sees that 
$S(\la_l \lan 1 \ran)$ has a composition factor $D^{\nu_l \lan i \ran^\flat}$ 
for any $i\in \{1,\cdots, \Fd_{\nu_l}^\flat\}$ 
(Note that $\Fd_{\la_l}=1$).  
On the other hand, 
by Proposition \ref{prop decom rel G(r,1,n) G(r,p,n)} (\roii), 
one sees that 
$S(\la_{l-1}\lan j \ran)$ ($1 \leq j \leq \Fd_{\la_{l-1}}$) 
has a composition factor 
$D^{\nu_l \lan i \ran^\flat}$ for some $i \in \{1,\cdots, \Fd_{\nu_l}^\flat\}$. 
Thus, 
we have that 
$S(\la_l \lan 1 \ran) \sim S(\la_{l-1} \lan j \ran)$ for any $j =1,\cdots, \Fd_{\la_{l-1}}$. 
This implies that 
$S(\la_{l-1} \lan 1 \ran) \sim S(\la_{l-1} \lan 2 \ran) 
	\sim \cdots \sim S(\la_{l-1} \lan \Fd_{\la_{l-1}} \ran)$. 
By using the (backword) inductive argument 
combined with Proposition \ref{prop decom rel G(r,1,n) G(r,p,n)}, 
we have the proposition.
\end{proof}

\begin{thm}\
\label{main thm}
\begin{enumerate}
\item 
For $\la \in \vG$ and $i=1,\cdots, \Fd_\la$, 
we have that 
$S(\la\lan i \ran)$ (resp. $\D(\la \lan i \ran)$) 
is a simple $\He$-module 
(resp. a simple object of $\CO$).  
Moreover, 
$S(\la\lan i \ran)$ (resp. $\D(\la \lan i \ran)$) 
is a block of $\He$ (resp. of $\CO$) itself.  
 
\item 
For $\la,\mu \in \CP_{n,r} \setminus \vG$ 
and $i,j \in \ZZ$, 
we have that 
\[
\D(\la \lan i \ran ) \sim \D(\mu \lan j \ran ) 
\Leftrightarrow  
S(\la \lan i \ran ) \sim S(\mu \lan j \ran) 
\Leftrightarrow \la \approx \mu.\]
\end{enumerate}
\end{thm}

\begin{proof}
Suppose that 
$S(\la \lan i \ran)$ and $S(\mu \lan j \ran)$ 
have a common composition factor 
$D^{\nu \lan k \ran^\flat}$. 
Then, by Proposition \ref{prop decom rel G(r,1,n) G(r,p,n)} (\roii), 
$S^{\dag \la[i']}$ and $S^{\dag \mu[j']}$ 
have a common composition factor 
$D^{\dag \mu}$ 
for some $i', j'$. 
This implies that 
\begin{align}
\label{S la sim S mu only if} 
S(\la \lan i \ran) \sim S(\mu \lan j \ran) 
\text{ only if } 
\la \approx \mu.
\end{align}

(\roi) 
Suppose that 
$\la \in \vG$, 
then 
$S^{\dag \la}$ is a simple $\He^\dag$-module 
from the definition of $\vG$. 
If 
$S(\la \lan i \ran) \sim S(\mu \lan j \ran)$ 
for some $\mu \lan j \ran \in \vL^+$,  
we have that $\la \approx \mu$ by \eqref{S la sim S mu only if}. 
This implies that 
there exists an integer $l$ 
such that 
$\la = \mu[l]$ 
since $\la \in \vG$.  
Thus, we have that 
$\mu \lan j \ran =\mu[l]\lan j\ran = \la \lan j \ran$ 
from the definition of $\vL^+$. 
Now we may assume that 
$S(\la \lan i \ran)$ and $S(\la \lan j \ran)$ 
have a common composition factor 
$D^{\mu \lan k \ran^\flat}$. 
If $\la \lan i \ran \not= \la \lan j \ran$ 
(i.e. $i \not\equiv j \mod \Fd_\la$), 
we have $\sum_{s=1}^{\Fd_\la} [ S(\la \lan s \ran) : D^{\mu \lan k \ran^\flat}]_{\He} \geq 2$.  
On the other hand, 
we have 
$\sum_{t=1}^{\Fk_\mu^\flat} [S^{\dag \la} : D^{\dag \mu[t]^\flat}]_{\He^\dag} \leq 1$ 
since $S^{\dag \la}$ is simple. 
These contradict to Proposition \ref{prop decom rel G(r,1,n) G(r,p,n)} (\roi). 
Thus we have $\la \lan i \ran = \la \lan j \ran = \mu \lan j \ran$. 
This implies (\roi).

Next we prove (\roii). 
For $\la,\mu \in \CP_{n,r} \setminus \vG$, 
suppose that 
$S^{\dag \la}$ and $S^{\dag \mu}$ 
have a common composition factor $D^{\dag \nu}$. 
Then, 
by Proposition \ref{prop decom rel G(r,1,n) G(r,p,n)} (\roi), 
$S(\la\lan i \ran)$ and $S(\mu \lan j \ran)$ 
have a common composition factor 
$D^{\nu\lan l \ran^\flat}$ for some $i,j$ (and for any $l$). 
Thus, 
$S(\la \lan i \ran) \sim S(\mu \lan j \ran)$. 
Combining Proposition \ref{prop sim sim sim} and \eqref{S la sim S mu only if}, 
we obtain the theorem.  

\end{proof}
%%%%%%%%%%%%%%%%%%%%%%%%%%%%%%%%%%%%%%%%%%%%%%%%%%%%%%%%%%%%%%%%%%%%%%%%%%%%%%%%%%%%%%%%%%%%%%%%%%%%%%%%%%%%%%%%%%%%%

%%%%%%%%%%%%%%%%%%%%%%%%%%%%%%%%%%%%%%%%%%%%%%%%%%%%%%%%%%%%%%%%%%%%%%%%%%%%%%%%%%%%%%%%%%%%%%%%%%%%%%%%%%%%%%%%%%%%%

%%%%%%%%%%%%%%%%%%%%%%%%%%%%%%%%%%%%%%%%%%%%%%%%%%%%%%%%%%%%%%%%%%%%%%%%%%%%%%%%%%%%%%%%%%%%%%%%%%%%%%%%%%%%%%%%%%%%%

%%%%%%%%%%%%%%%%%%%%%%%%%%%%%%%%%%%%%%%%%%%%%%%%%%%%%%%%%%%%%%%%%%%%%%%%%%%%%%%%%%%%%%%%%%%%%%%%%%%%%%%%%%%%%%%%%%%%%

\end{document}